\documentclass[12pt]{amsart}
\usepackage{amsfonts}
\usepackage[cp1250]{inputenc}
\usepackage{amssymb}
\usepackage{amsmath}
\usepackage{amsthm}
\usepackage{latexsym}
\usepackage{amsbsy}
\usepackage{graphicx}
\usepackage{color}
\usepackage{boxedminipage}
\usepackage{geometry}
\usepackage[cp1250]{inputenc}
\usepackage{graphicx}
\usepackage{cite}

\newtheorem{theorem}{Theorem}

\newtheorem*{theoremS}{Sarason's Theorem}
\theoremstyle{plain}
\newtheorem{prop}{Proposition}

\newtheorem*{theo}{Theorem}
\newtheorem{lemma}{Lemma}

\theoremstyle{definition}

\newtheorem*{rems}{Remarks}
\theoremstyle{definition}

\DeclareMathOperator{\id}{id} 
\DeclareMathOperator{\im}{Im}

\def\C{\mathbb{C}}%
\def\D{\mathbb{D}}%

\def\H{\mathcal{H}}%
\def\M{\mathcal{M}}%

\begin{document}

\title{The complement of $\M(a)$ in $\H(b)$ }

\author{M. T. Nowak, P. Sobolewski and A. So\l tysiak }


 \maketitle

\begin{abstract}
Let $b$ be a nonextreme function in the unit ball of $H^{\infty}$ on the unit disk
$\mathbb D $ and let  $a$ be an outer $H^{\infty}$ function such
that $|a|^2+|b|^2=1$ almost everywhere on $\partial \D$. The
sufficient and necessary conditions for the orthogonal complement of
$\M(a)$ in $\H(b)$ be finite dimensional  has been given by
D. Sarason in \cite{Sarason4}. Here we describe this space
explicitly.
\end{abstract}

\section{Introduction}
Let $H^2$ denote the standard Hardy space on the unit disk $\D$ and
let $\partial\D$ denote its boundary. For $\varphi\in
L^{\infty}(\partial\D)$ the Toeplitz operator on $H^2$ is given by
$T_{\varphi}f=P_{+}(\varphi f)$, where $P_{+}$ is the orthogonal
projection of $L^2(\partial\D)$ onto $H^2$. For a nonconstant
function $b$ in the unit ball of $H^\infty$ the de Branges-Rovnyak
space $\H(b)$  is the image of  $H^2$ under  the operator
$(1-T_bT^*_b)^{1/2}$ with the corresponding range norm. It is known
\cite[p.10]{Sarason4} that $\H(b)$ is a Hilbert space with
reproducing kernel
\[
k_w^b(z)=\frac{1-\overline{b(w)}b(z)}{1-\bar{w}z}\quad(z,w\in\D).
\]

Here we are interested in the case when the function $b$ is not an
extreme point of the unit ball of $H^{\infty}$, that is the case
when the function $\log(1-|b|)$ is integrable on $\partial\D$
(\cite[p. 138]{hoffman}). Then there exists an outer function $a\in
H^{\infty}$ for which $|a|^2+|b|^2=1$ a.e. on $\partial\D$.
Moreover, if we suppose that $a(0)>0$, then $a$ is uniquely
determined, and we say that $(b,a)$ is a pair. Since the function
$\frac {1+b}{1-b}$ has a positive real part,
there exists a positive measure
$\mu$ on $\partial\D$ such that
\begin{equation}
\frac{1+b(z)}{1-b(z)}=\int_{\partial\D}\frac {1+ e^{-i\theta}z}{1-
e^{-i\theta}z}\,d\mu(e^{i\theta})+i\im\frac{1+b(0)}{1-b(0)},\quad
|z|<1. \label{Herglotz}
\end{equation}
Moreover the function
$\left|\frac{a}{1-b}\right|^2$ is the Radon-Nikodym derivative of
the absolutely continuous component of $\mu$ with respect to the
normalized Lebesgue measure. So if $f=\frac a{1-b}$, then $f$ is an
outer function which belongs to $H^2$.  If the measure $\mu$ is
absolutely continuous the pair $(b,a)$ is called \emph{special}. The
operator $V_b$ acting on $H^2(\mu)$ (the closure of polynomials
in $L^2(\mu)$) with values in the Branges-Rovnyak space $\H(b)$ is
given by
\begin{equation}
(V_bq)(z)=(1-b(z))\int_{\partial\D}\frac{q(e^{i\theta})}{1-e^{-i\theta}z}\,d\mu(e^{i\theta}).
\label{Vb}
\end{equation}
It is known that $V_b$ is an isometry of $H^2(\mu)$ onto $\H(b)$
(\cite{Sarason4}, \cite{FM}).  Furthermore the Toeplitz operators
with an unbounded symbols $\varphi\in L^2(\partial\D)$ can be
defined as unbounded operators on $H^2$ (with the domains containing
$H^{\infty}$)  that are continuous operators of $H^2$ into
$H(\D)$, the space of holomorphic functions on $\D$ with the
topology of the locally uniform convergence. Moreover, if
$(b,a)$ is a pair and
$f=\frac{a}{1-b}$, then the operator $T_{1-b}T_{\bar{f}}$ is an
isometry of $H^2$ into $\H(b)$. Its range is all of $\H(b)$ if and
only if the pair $(b,a)$ is special (see \cite[IV-12,13]{Sarason4}).

If  $(b,a)$ is a pair then $\M(a)$ (the range of
$T_aH^2=aH^2$ equipped with  the range norm) is contained
contractively in $\H(b)$. Moreover, if $(b,a)$ is
special, then  $\M(a)$ is dense in $\H(b)$ if and
only if $f^2$ is a rigid function. Recall that a function  $f\in
H^1$  is called \emph{rigid} if no other functions in
$H^1$, except for positive scalar multiples of $f$, have the same
argument as $f$ a.e. on $\partial\D$.

Let the Toeplitz operator $ T_{z}$, that is, the unilateral shift on
$H^2$, be denoted by $S$. It is known that the  de Branges-Rovnyak
spaces $\H(b)$  are $S^*$-invariant.  In the case $b$ is
nonextreme the space $\H(b)$ is also invariant under the
unilateral shift $S$. Furthermore, in this case, polynomials are
dense in $\H(b)$ and $ b\in \H(b)$.

Let $\H_0(b)$ denote the
orthogonal complement of $\M(a)$ in $\H(b)$. Let  $Y$ be the
restriction of the shift operator $S$ to $\H(b)$. It is worth to
mention here that since the closure of $\M(a)$ in $\H(b)$ is  $Y$-invariant,
the space $\H_0(b)$ is $Y^*$-invariant. Let $Y_0$
be the compression of $Y$ to the subspace $\H_0(b)$.
Characterizations when $\H_0(b)$ has finite dimension are given in
Chapter X of \cite{Sarason4}. Now we cite some results included
therein.  It turns out that in this case the space $\H_0(b)$ depends
on the spectrum of the restriction of the operator $Y^*$ to
$\H_0(b)$ which actually equals $Y^*_0$. The spectrum of $Y_0$ is
contained in the unit circle. We know  from \cite{Sarason4} that the
codimension of $\overline{\M(a)}$ in $\H(b)$ is $N$  if  and only if
the operator $Y^*_0$ has eigenvalues $z_1, z_2,\dots, z_s$ on the
unit circle with their algebraic multiplicities $n_1,\dots, n_s$ and
$N=n_1+n_2+\dots+n_s$. Then
\[
\H_0(b)=\bigoplus_{j=1}^s \ker(Y^*_0-\bar
{z}_j)^{n_j}.
\]

For  $\lambda\in \partial \D$ let
$\mu_{\lambda}$ be the measure on $\partial \D$ whose  Poisson
integral is the real part of $\frac
{1+\overline{\lambda}b}{1-\overline{\lambda}b}$. If
$F_{\lambda}=\frac{a}{1-\overline{\lambda}b}$, then the Radon-Nikodym
derivative of the absolutely continuous component of $\mu_{\lambda}$
is $|F_{\lambda}|^2$.

In \cite{Sarason4} the following condition for
$\M(a)$ to have a finite defect is given.

\begin{theo} Let $N$ be a positive integer, and let $\lambda$ be a point on
$\partial \D$ such that the measure $\mu_{\lambda}$ is absolutely
continuous. Then the following conditions are equivalent.
\begin{itemize}
\item[(i)]
The codimention of
 $\overline{\M(a)}$ in $\H(b)$ is $N$.
\item[(ii)]
$F_{\lambda}=pf$, where $p$ is a polynomial of degree $N$
having all of its roots on the unit circle, and $f$ is a function in
$H^2$ whose square is rigid.
\end{itemize}
\end{theo}
Our aim is to find  explicit description of  finite dimensional
spaces $\H_0(b)$.

For $w\in \D$ and a positive integer $n$ the function
$\frac{\partial^nk_{w}^b}{\partial \bar{w} ^n}$ is the kernel
function in $\H(b)$  for the functional of evaluation of the $n$-th
derivative at $w$, that is, for $f\in\H(b)$, we have
\[
f^{(n)}(w)= \biggl\langle f,\frac{\partial^nk_{w}^b}{\partial \bar{w}
^n}\biggr\rangle.
\]

For $n=0,1,2,\ldots$ set
\[
\upsilon^n_{b,w}(z)=\frac{\partial^nk_w^b}{\partial\bar{w}^n}(z),
\quad z,w\in\D.
\]
Our main result is the following.

\begin{theorem}
Assume that a point $z_0\in\partial\D$ and
for $\lambda\in\partial\D\setminus\{b(z_0)\}$ the measure
$\mu_{\lambda}$ is absolutely continuous. If  the function
$F_{\lambda}(1-\bar{z}_0 z)^{-k-1}$ is in $H^2$, then the space
$\ker(Y^*-\bar{z}_0)^k$ is spanned by $\upsilon^0_{b,z_0}$,
$\upsilon^1_{b,z_0}$,  \ldots, $\upsilon^k_{b,z_0}$ which are the
limits of $\upsilon^0_{b,w}$, $\upsilon^1_{b,w}$, \ldots,
$\upsilon^k_{b,w}$ as $w$ tends nontangentially to $z_0$.
\end{theorem}
We mention that this theorem for $k=1$  has been proved in
\cite{NSSW}.

\section{Preliminaries}
In this section we collect auxiliary
results on the space $\H(b)$ generated by a nonextreme $b$ that we
will use in our proofs. Let $X$ denote the restriction of the
operator $S^*$ to $\H(b)$. Then the adjoint operator $X^*$
is given by
\begin{equation}
X^*h= Sh-\langle h, S^*b\rangle b, \qquad \text{(see
\cite[p. 61]{Sarason4}, \cite[Theorem 18.22]{FM}).}\label{adjoint}
\end{equation}
Moreover,  the
following formula for the  reproducing kernel $k_w^b$
was given in\cite{Sarason0} (see also
Theorem 18.21 in \cite{FM})
\begin{equation}
k_w^b=(1-\bar wX^*)^{-1}k_0^b.
\label{kernel1}
\end{equation}
Using this formula we derive the following.
\begin{prop}
If $\frac{\partial^nk_{w}^b}{\partial \bar{w} ^n}$   are bounded in the norm as $w$ tends nontangentially to
$z_0\in\partial\D$, then also the norms of
$\frac{\partial^mk_{w}^b}{\partial \bar{w}^m}$, $m=0,1, \dots ,n-1,
$ stay bounded as $w$ tends nontangentially to $z_0$.\label{bound}
\end{prop}

\begin{proof}
Formula (\ref{kernel1}) implies that for $n=1,2,\dots$
\[
\frac{\partial^nk_{w}^b}{\partial \bar{w} ^n}= n!(1-\bar
wX^*)^{-n-1}X^{*n}k_0^b.
\]
Since
\[
(1-\bar wX^*)^{-n}X^{*n}k_0^b=(1-\bar{w}X^*)(1-\bar
wX^*)^{-n-1}X^{*n}k_0^b,
\]
the boundedness  of $(1-\bar{w}X^*)^{-n-1}X^{*n}k_0^b$ implies the
boundedness of $(1-\bar{w}X^*)^{-n}X^{*n}k_0^b$. To show that
\[
\frac{\partial^{n-1}k_{w}^b}{\partial \bar w ^{n-1}}=(n-1)!(1-\bar{w}X^*)^{-n}X^{* (n-1)}k_0^b
\]
is bounded we observe that
\[
X^*\frac{\partial^{n-1}k_{w}^b}{\partial \bar{w} ^{n-1}}=
(n-1)!(1-\bar{w}X^*)^{-n}X^{* n}k_0^b. \] It follows from
(\ref{adjoint}) that
\[
\|X^*f\|^2=\|f\|^2-|\langle f, S^*b\rangle|^2 \qquad \text{ (see
Corollary 18.23 in \cite{FM})}
\]
We  know that if $b$ is
nonextreme, then $b\in \H(b)$. Thus
\begin{align*}
\left\| \frac1{(n-1)!}\frac{\partial^{n-1}k_{w}^b}{\partial \bar{w}
^{n-1}}\right\|^2&=\|(1-\bar{w}X^*)^{-n}X^{*n}k_0^b\|^2+|\langle(1-\bar{w}X^*)^{-n}X^{*(n-1)}k_0^b,
S^*b \rangle|^2\\ &=  \|(1-\bar{w}X^*)^{-n}X^{*n}k_0^b\|^2+|\langle(1-\bar{w}X^*)^{-n-1}k_0^b, X^{n}b
\rangle|^2\\ \noalign{\medskip}
&=\|(1-\bar{w}X^*)^{-n}X^{*n}k_0^b\|^2+|\langle(1-\bar{w}X^*)^{-n-1}X^{*n}k_0^b, b
\rangle|^2.
\end{align*}
\end{proof}

The next proposition was actually stated in \cite[pp. 58--59]{Sarason4}
without proof.

\begin{prop}
Let  $z_0\in \partial \D$ and  $b$
be a nonextreme function from the unit ball of $H^{\infty}$. If
every function in $\H(b)$ and all of its derivatives up to
order $n$ have nontangential limits at $z_0$, then also  $b^{(n+1)}$
has a nontangential limit at $z_0$.
\end{prop}

\begin{proof}
The case $n=0$ is contained in  \cite[VI-4]{Sarason4}
Assume that   every
function in $\H(b)$ and all of its derivatives up to order
$n$  have nontangential limits at $z_0$. It follows from
\cite[VI-4]{Sarason4} that then the function
$h(z)=\frac{b(z)-b(z_0)}{z-z_0}$ is in $\H(b)$. Thus there
exists the limit
\[
\lim_{\substack{z\to z_0 \\
\angle}}\left(\frac{b(z)-b(z_0)}{z-z_0}\right)^{(n)}=\lambda.
\]
For $C\geq 1$, let the Stolz domain $S_{C}(z_0)$ be  defined by
\[
S_{C}(z_0)=\{z\in\D: |z-z_0|\leq C(1-|z|)\},
\]
and for an  $\varepsilon>0$, put
\[
\alpha_n(\varepsilon)=\sup\left\{\left|\left(\frac{b(z)-b(z_0)}{z-z_0}\right)^{(n)}-\lambda\right|:
z\in S_{2C+1}(z_0), |z-z_0|<\varepsilon\right\},
 \]
then clearly
$\alpha_n(\varepsilon)$ tends to zero with $\varepsilon$.

Next let $\gamma_z$ denote the circle with center $z$ and radius
$\frac12(1-|z|)$. Then for $z\in S_C(z_0)$,
$\gamma_z$ lies in $S_{2C+1}(z_0)$.
The Leibniz formula
\[
\left(\frac{b(z)-b(z_0)}{z-z_0}\right)^{(n)}=\sum_{k=0}^{n}\binom
nk\left(b(z)-b(z_0)\right)^{(n-k)}\frac{(-1)^kk!}{(z-z_0)^{k+1}}
\]
implies that for $\zeta \in\gamma_z$,
\[
b^{(n)}(\zeta)=\sum_{k=1}^{n}\binom
nk\frac{(-1)^{k-1}k!(b(\zeta)-b(z_0))^{(n-k)}}{(\zeta-z_0)^{k}} +
\lambda (\zeta-z_0)+\beta(\zeta)(\zeta-z_0),
\]
where
$|\beta(\zeta)|\leq\alpha_n(|\zeta- z_0|)\leq\alpha_n\left(\frac 32|z-
z_0|\right)$. Using the equality  $k\binom n{k}=n\binom{n-1}{k-1}$,
we obtain
\begin{align*}
b^{(n)}(\zeta) &=\sum_{k=1}^{n}k\binom n k
(b(\zeta)-b(z_0))^{(n-k)}\left(\frac{1}{\zeta-z_0}\right)^{(k-1)} +
\lambda
(\zeta-z_0)+\beta(\zeta)(\zeta-z_0)\\
&=\sum_{k=1}^{n}n\binom
{n-1}{k-1}(b(\zeta)-b(z_0))^{(n-k)}\left(\frac{1}{\zeta-z_0}\right)^{(k-1)}
+ \lambda
(\zeta-z_0)+\beta(\zeta)(\zeta-z_0)\\
&=n\sum_{k=0}^{n-1}\binom
{n-1}{k}(b(\zeta)-b(z_0))^{(n-1-k)}\left(\frac{1}{\zeta-z_0}\right)^{(k)}
+ \lambda
(\zeta-z_0)+\beta(\zeta)(\zeta-z_0)\\
&=n\left(\frac{b(\zeta)-b(z_0)}{\zeta-z_0}\right)^{(n-1)} + \lambda
(\zeta-z_0)+\beta(\zeta)(\zeta-z_0).
\end{align*}
Finally, since
\[
b^{(n+1)}(z)=\frac1{2\pi
i}\int_{\gamma_z}\frac{b^{(n)}(\zeta)}{(\zeta-
 z)^2}d\zeta,
\]
we get
\[
\lim_{\substack{z\to z_0\\z\in S_{2C+1}(z_0)}}b^{(n+1)}(z)=(n+1)\lambda.
\]
\end{proof}

For $\lambda\in\partial\D$  set
$W_{\lambda}=T_{1-\overline{\lambda}b}T_{\overline {F}_{\lambda}}$
and recall that if $\mu_{\lambda}$ is absolutely continuous then
$W_{\lambda}$ is an isometry of $H^2$ onto $\H(b)$. In \cite
{Sarason4} the structure of finite dimensional spaces $\H_0(b)$ has
been studied by means of an operator
$A_\lambda$ on $H^2$. Under the assumption that $\mu_{\lambda}$ is absolutely continuous, $A_\lambda$
intertwines
$W_{\lambda} $  with the operator $Y^\ast$ i.e.,
\begin{equation}
W_{\lambda}A_\lambda=Y^* W_{\lambda}.\label{inter}
\end{equation}
The operator $A_\lambda$ is given by
\begin{equation}
A_\lambda=S^*-F_\lambda(0)^{-1}(S^* F_\lambda\otimes 1).
\label{al}\end{equation}
 Moreover,  it has been showed in \cite[X-14]{Sarason4}
 that under above assumptions, if for $z_0\in\partial \D$ the function $\frac{F_{\lambda}}{(1-\bar{z}_0z)^{k}}\in H^{2}$, then  the kernel of
$(A_{\lambda}-\bar{z}_0)^k$ is spanned by
$(1-\bar{z}_0z)^{-1}F_{\lambda},
(1-\bar{z}_0z)^{-2}F_{\lambda},\dots,
(1-\bar{z}_0z)^{-k}F_{\lambda}$. It turns out the the inverse
statement is also true and we have

\begin{prop}
Assume that $z_0\in\partial\D$ and $\lambda\in
\partial\D$ is such that $\mu_{\lambda}$ is absolutely continuous.
Then
\[
\dim\ker(A_{\lambda}-\bar{z_0})^k= k \iff
F_{\lambda}(1-\bar{z_0}z)^{-k}\in H^2.
\]
\end{prop}
\begin{proof}
 In view of \cite[X-14]{Sarason4} it is
enough to show
\[
\dim\ker(A_{\lambda}-\bar{z}_0)^k= k\implies
F_{\lambda}(1-\bar{z}_0z)^{-k}\in H^2.
\]
We will show that if $
\dim\ker(A_{\lambda}-\bar{z}_0)^k=k$, then $
\ker(A_{\lambda}-\bar{z}_0)^k $ is spanned by $\frac
{F_{\lambda}}{1-\bar{z}_0z},$ $ \frac
{zF_{\lambda}}{(1-\bar{z}_0z)^2}, \dots
\frac{z^{k-1}F_{\lambda}}{(1-\bar{z}_0z)^k}$. We proceed by
induction. The case $k=1$ is proved in \cite[X-13]{Sarason4}. Suppose that
\[
\ker(A_{\lambda}-\bar{z}_0)^k=\left\{ \frac {c_0F_{\lambda}}{1-\bar{z}_0z}
+ \frac {c_1zF_{\lambda}}{(1-\bar{z}_0z)^2}+\dots
+\frac{c_{k-1}z^{k-1}F_{\lambda}}{(1-\bar{z}_0z)^k}\colon\,c_0,c_1,c_2,\ldots,c_{k-1}\in\C\right\}
\]
and $(A_{\lambda}-\bar{z_0})g \in
\ker(A_{\lambda}-\bar{z_0})^{k}$. Then
\[
(A_{\lambda}-\bar{z_0})g = \frac {c_0F_{\lambda}}{1-\bar{z}_0z} +
\frac {c_1zF_{\lambda}}{(1-\bar{z}_0z)^2}+\dots
+\frac{c_{k-1}z^{k-1}F_{\lambda}}{(1-\bar{z}_0z)^k}
\]
for some  $c_0,c_1,c_2,\ldots,c_{k-1}$. Hence, by (\ref{al}),
\[
\frac{g(z)-g(0)}{z}- \frac{g(0)}{F_{\lambda}(0)}\frac
{F_{\lambda}(z)-F_{\lambda}(0)}{z} -\bar{z}_0g= \frac
{c_0F_{\lambda}}{1-\bar{z}_0z} + \frac
{c_1zF_{\lambda}}{(1-\bar{z}_0z)^2}+\dots
+\frac{c_{k-1}z^{k-1}F_{\lambda}}{(1-\bar{z}_0z)^k} \] or,
equivalently,
\[
(1-\bar{z}_0z)g= \frac{g(0)}{F_{\lambda}(0)}F_{\lambda}+ \frac
{c_0zF_{\lambda}}{1-\bar{z}_0z} +\dots
+\frac{c_{k-1}z^{k}F_{\lambda} }{(1-\bar{z}_0z)^k}.
\]
\end{proof}
The next lemma will allow  us to depict  $\ker (Y^*-\bar{z}_0)^k$
explicitly.
\begin{lemma}  If   $\mu_{\lambda}$ is absolutely continuous, then
for any positive integer $k$ and any $z_0\in\partial\D$,
\begin{equation}
\ker(Y^*-\bar{z}_0)^k= W_{\lambda}\ker
(A_{\lambda}-\bar{z}_0)^k.\label{Tlambda}
\end{equation}
\end{lemma}

\begin{proof}
Since $W_{\lambda}$ is an isometry of $H^2$ onto $\H(b)$, we  have
$W_{\lambda}^* W_{\lambda}=\id$ on $H^2$. Indeed, for any $f\in
H^2$,
\[
\langle f,f\rangle_{H^2}=\langle
W_{\lambda}f,W_{\lambda}f\rangle_{{\H}(b)}=\langle W_{\lambda}^*
W_{\lambda}f, f\rangle_{H^2}.
\]

Next, if $g\in\H(b)$  is the image of  $f\in H^2$ under $W_\lambda$,
then  $W^*_\lambda W_\lambda f=f$ implies $W_\lambda
W_{\lambda}^* g=g$, i.e. $W_{\lambda}W_{\lambda}^*=\id$ on
$\H(b)$.

Note that (\ref{inter}) implies
\[
W_{\lambda}(A_{\lambda}-\bar{z}_0)=(Y^* - \bar{z}_0)W_{\lambda}.
\]
Since $W_\lambda W^*_\lambda=\id$ on $\H(b)$ we get
\[
(Y^{*}-\bar{z}_0)=W_{\lambda}(A_{\lambda}-\bar{z}_0)W_{\lambda}^{*}
\]
and, by iteration,
\begin{equation}
(Y^{*}-\bar{z}_0)^k=W_{\lambda}(A_{\lambda}-\bar{z}_0)^k\label{sar}
W_{\lambda}^{*}.
\end{equation}

Assume that   $g\in\H(b) $ and $g=W_{\lambda}f$  $(f\in H^2)$ and
observe that by (\ref{sar}),  $g=W_{\lambda}f
\in\ker(Y^{*}-\bar{z}_0)^k$ if and only if
\[
0=(Y^{*}-\bar{z}_0)^kg=W_{\lambda}(A_{\lambda}-\bar{z}_0)^k
W_{\lambda}^{*}W_{\lambda}f=W_{\lambda}(A_{\lambda}-\bar{z}_0)^kf
\]
which means that $f\in \ker(A_{\lambda}-\bar{z}_0)^k$.
\end{proof}

\section{Proof of Theorem 1}
In the proof we  will use the following result stated in Chapter VII
in \cite{Sarason4}.

\begin{theoremS} Assume that   a point $z_0\in\partial\D$ and
for $\lambda\in\partial\D\setminus\{b(z_0)\}$ the measure
$\mu_{\lambda}$ is absolutely continuous. Then the following
conditions are equivalent.
\begin{enumerate}
\item[(i)] Each function in $\H(b)$ and all of its
derivatives up to order $k$ have nontangential limits at $z_0$.
\item[(ii)] The function $F_{\lambda}(1-\bar{z}_0 z)^{-k-1}\in H^2$.
\item[(iii)] The functions $\frac{\partial^{k}k_w^b}{\partial\bar{w}^{k}}$
are bounded in the norm as $w$ tends nontangentially to $z_0$.
\end{enumerate}
\end{theoremS}

The proof of this theorem for  the case when $k=0$ is  given in
\cite [VI-4]{Sarason4}.  Since  the proof of the general case is
only sketched in the cited reference, we include it here
for the reader's convenience.

\begin{proof}[Proof of Sarason's Theorem] (i)$\implies$(iii). Since for $f\in\H(b)$
\[
f^{(k)}(w)=\biggl\langle f,\frac{\partial^{k}
k^b_w}{\partial\bar{w}^k}\biggr\rangle, \quad w\in\D
\]
and  the nontangential limit of $f^{(k)}(w) $ at $z_0$  exists,
$\sup\{|f^{(k)}(w)|\colon\,w\in S_C(z_0)\}$ is finite. Let
$\varphi_w(f)=f^{(k)}(w)$ be a bounded linear functional on $\H(b)$.
The Banach-Steinhaus theorem implies that there exists a constant
$M>0$ such that
\[
\|\varphi_w\|=\biggl\|\frac{\partial^{k}k_w^b}{\partial\bar{w}^{k}}\biggr\|\leqslant
M
\]
for every $w\in S_C(z_0)$.

(iii)$\implies$(i). We proceed by induction. Assume the implication holds true  for
 $k=0,1, \dots, m-1$ and suppose that   the functions $\frac{\partial^{m}k_w^b}{\partial\bar
{w}^{m}}$ are bounded in the norm as $w$ tends nontangentially to
$z_0$. Then there exists  a sequence $\{w_n\}\subset\D$ that
converges nontangentially to $z_0$ for which
$\left\{\frac{\partial^m k_{w_n}^b}{\partial\bar{w}_n^m}\right\}$
converges weakly to $h\in\H(b)$.

Thus  we have
\[
 h(z)=\langle h, k_z^b\rangle=\lim_{n\to\infty}
\biggl\langle\frac{\partial^m k_{w_n}^b}{\partial \bar{w}_n^m},
k_z^b\biggr\rangle=\lim_{n\to\infty}\frac{\partial^m
k_{w_n}^b}{\partial \bar{w}_n^m}(z).
\]

Since
\begin{equation}
\frac{\partial^{m}k_w^b}{\partial\bar{w}^{m}}(z) =
\sum_{j=0}^{m}\binom{m}{j}\frac{\partial^{m-j}(1- \overline
{b(w)}b(z))}{\partial\bar{w}^{m-j}} \frac{j! z^j}{(1-\bar{w}
z)^{j+1}}, \label{mder}
\end{equation}
we have
\begin{align}\label{bderiv}
h(z)&=\lim_{n\to\infty}\frac{\partial^m k_{w_n}^b}{\partial \bar{w}_n^m}(z)\\
&=\lim_{n\to\infty}\frac{(1-\overline{b(w_n)}b(z))m!z^m}{(1-
\bar{w}_nz)^{m+1}}+
\sum_{j=0}^{m-1}\binom{m}{j}\lim_{n\to\infty}\frac{-\overline
{b^{(m-j)}(w_n)}b(z)j! z^j}{(1-\bar{w}_n z)^{j+1}}. \nonumber
\end{align}
Furthermore Proposition 1 implies that the norms
$\|\frac{\partial^{k}k_w^b}{\partial\bar{w}^{k}}\|, k=0,1,\dots,
m-1$, stay also bounded as $w$ tends nontangentially to $z_0$, .
Hence by induction hypothesis each function in $\H(b)$ and all of
its derivatives up to order $m-1$ have nontangential limits at
$z_0$. Additionally, by Proposition 2,  the derivative of order $m$
of $b$  has its nontangential limit at $z_0$. It then follows from
(\ref{bderiv}) that $\frac{\partial^{m}k_w^b}{\partial\bar {w}^{m}}$
converges to $h$ pointwise  as $w$ tends to $ z_0$ nontangentially.
Put $h= \frac{\partial^{m}k_{z_0}^b}{\partial\bar{z}_0^{m}}$ and
note that for $z\in\D$,
\[
\lim_{\substack{w\to z_0 \\ \angle}}(k_z^b)^{(m)}(w)=\overline{\frac
{\partial^{m} k_{z_0}^b}{\partial\bar{z}_0^m}(z)}.
\]
Indeed,
\[
\lim_{\substack{w\to z_0 \\ \angle}}(k_z^b)^{(m)}(w)=\lim_{\substack{w\to z_0 \\
\angle}}\biggl\langle k_z^b,\frac{\partial^m k_w^b}{\partial \bar
{w}^m}\biggr\rangle=\lim_{\substack{w\to z_0 \\
\angle}}\overline{\frac{\partial^m k_w^b}{\partial \bar{w}^m}(z)}=\overline{\frac {\partial^m
k_{z_0}^b}{\partial\bar{z}_0^m}(z)}=(k_z^b)^{(m)}(z_0) .
\]
This means that for every $z\in\D$, the derivative $(k_z^b)^{(m)}$
has a nontangential limit at $z_0$. Since the functions $k_z^b$ span
the space $\H(b)$ and the norms of
$\frac{\partial^{m}k_w^b}{\partial\bar {w}^{m}}$ are bounded as $w$
tends nontangentially to $z_0$, the desired  conclusion follows.

(ii)$\implies$(iii).  Assume that the implication holds for
$k=0,1,\dots, m-1$ and  $F_{\lambda}(1-\bar z_0 z)^{-m-1}\in H^2$
Observe first that if  $F_{\lambda}(1-\bar z_0 z)^{-m-1}\in H^2$,
then also $F_{\lambda}(1-\bar z_0 z)^{-k-1}\in H^2$ for
$k=0,1,\dots, m $.

Let $V_{\bar{\lambda}b}$ be defined by (\ref{Vb}) with $b$ replaced
by $\bar{\lambda} b$. Since the pair $(a,\lambda b)$ is special,   we have
\[
V_{\bar{\lambda}b}(1-\lambda\overline
{b(w)})k_{w}=W_\lambda(F_{\lambda}(1-\lambda\overline
{b(w)})k_w)=k_{w}^b,\quad w\in\D,
\]
(see \cite[p.18]{Sarason4}, \cite[Vol.2, p.141]{FM}).  Since
\begin{equation}
W_\lambda\left(\frac{\partial^{m}((1-\lambda \overline
{b(w)})k_w)}{\partial\bar{w}^m}F_{\lambda}
\right)=\frac{\partial^{m}k_w^b}{\partial\bar{
w}^{m}}.\label{kderivative}
\end{equation}
and
\begin{equation}
\frac{\partial^{m}((1-\lambda \overline
{b(w)})k_w)}{\partial\bar{w}^{m}} =
\sum_{j=0}^{m}\binom{m}{j}\frac{\partial^{m-j}(1-\lambda \overline
{b(w)})}{\partial\bar{w}^{m-j}} \frac{j! z^j}{(1-\bar{w}
z)^{j+1}}\label{kkder}
\end{equation}
we see that the preimage of
$\frac{\partial^{m}k_w^b}{\partial\bar{w}^{m}}$ is a linear
combination of the functions
$\frac{F_{\lambda}}{(1-\bar{w}z)^{j+1}}$, $j=0,1,2,\dots, m$, whose
coefficients depend on $\overline{b(w)}, \overline{b'(w)},\dots
,\overline{b^{(m)}(w)}$. By the induction hypothesis
$\frac{\partial^{m-1}k_w^b}{\partial\bar{w}^{m-1}}$ are bounded as
$w$ tends nontangentially to $z_0$. By what we have already proved,
this implies the existence of the limits of $\overline{b(w)},
\overline{b'(w)},\dots,\overline{b^{(m)}(w)}$ as  $w\to z_0$
nontangentially.  Now (\ref{kderivative}), (\ref{kkder}) and the
fact that $W_{\lambda} $ is an isometry imply that the norms of
$\frac{\partial^{m}k_w^b}{\partial\bar w^{m}}$ stay  bounded as $w$
converges nontangentially to $z_0$.

(iii)$\implies$(ii). Assume that the implication holds true for
$k=0, 1,\dots, m-1$ and $\frac{\partial^{m}k_w^b}{\partial\bar{
w}^{m}}$ are bounded in the norm as $w$ tends nontangentially to
$z_0$. Since (iii) is equivalent to (i),  the nontangential limits
 of $\overline{b(w)},
\overline{b'(w)},\dots,\overline{b^{(m)}(w)}$ as $w\to z_0$ exist.
By the induction hypothesis the functions
$\frac{F_{\lambda}}{(1-\bar{z}_0z)^{k+1}}$, $k=0,1,\dots, m-1$ are
in $H^2$. Finally, passage to the limit in (\ref{kderivative}) and
(\ref{kkder}) as  $w\to z_0$ nontangentially and the induction
hypothesis show that $\frac{(1-\lambda
\overline{b(z_0)})F_{\lambda}}{(1-\bar{z}_0z)^{m+1}}$ is in $H^2$.
Since $\lambda\ne b(z_0)$ our claim follows.
\end{proof}

\begin{proof}[Proof of Theorem 1]

We know from \cite[X-14]{Sarason4} that for $k=0,1,2\dots$
$\ker(A_\lambda-\bar{z}_0)^{k+1}$ has dimension $k+1, $  and is
spanned by the functions $F_{\lambda}(1-\bar{z_0}z)^{-1}$,
$F_{\lambda}(1-\bar{z}_0z)^{-2}$, \ldots,
$F_{\lambda}(1-\bar{z}_0z)^{-k-1}$. By Lemma 1 the kernel of
$(Y^{*}-\bar{z}_0)^k$ is spanned by the images of these functions
under $W_{\lambda}$. So, it is enough to show that for $k=0,1,\dots
$, $W_{\lambda}
\left(\frac{F_{\lambda}}{(1-\bar{z}_0z)^{k+1}}\right)$ is a linear combination of
$\upsilon^0_{b,z_0}$, $\upsilon^1_{b,z_0},  \ldots,
\upsilon^k_{b,z_0}$ defined in Introduction. We proceed by
induction. The case when $k=0$ has been proved in\cite{NSSW}. Assume
that the statement is true for $k=0,1,\dots, m-1$ and
$F_{\lambda}(1-\bar{z}_0 z)^{-m-1}$ is in $H^2$. Let $\{w_n\}\subset
\D$  be a sequence converging nontangentially to $z_0$. Then
\begin{align*}
{}&W_\lambda\left(\frac{\partial^{m}(1-\lambda \overline
{b(w_n)})k_{w_n}}{\partial\bar {w}_n^m}F_{\lambda} \right)\\
&=W_\lambda\left(\frac {(1-\lambda\overline{ b(w_n)})
m!z^mF_{\lambda}}{(1-\bar{w}_n z)^{m+1}}\right)-\sum_{j=1}^{m}\binom
mjW_{\lambda}\left(\frac{\lambda
\overline{b^{(j)}(w_n)}(m-j)!z^{m-j}F_{\lambda}}{(1-\bar{w}_nz)^{m-j+1}}\right)\\
&=\upsilon^m_{b,w_n}.
\end{align*}
Our assumption implies that
$\lim_{n\to\infty}b^{(j)}(w_n)=b^{(j)}(z_0),\  j=0,\dots, m$ and the
dominated convergence theorem  implies that the norm  of the each
function the operator $W_{\lambda}$ is acting on converges, as
$n\to\infty$, to a norm of a function in $H^2$ (which is of the form
$\frac{c_kF_{\lambda}}{(1-\bar{z}_0z)^{k+1}}, k=0,1,\dots, m$ and
$c_m\ne 0$). So the passage to the limit yields
\[
W_{\lambda}\left(\frac {(1-\lambda\overline{ b(z_0)})
m!z^mF_{\lambda}}{(1-\bar{z}_0z)^{m+1}}\right)-\sum_{j=1}^{m}\binom
mjW_{\lambda}\left(\frac{\lambda
\overline{b^{(j)}(z_0)}(m-j)!z^{m-j}F_{\lambda}}{(1-\bar{z}_0z)^{m-j+1}}\right)=\upsilon^m_{b,z_0}.
\]
Since by the induction hypothesis the second term  on the left-hand
side of the last equality is a linear combination of
$\upsilon^0_{b,z_0}$, $\upsilon^1_{b,z_0},  \ldots,
\upsilon^{m-1}_{b,z_0}$ our proof is finished.
\end{proof}

\begin{rems} The complement of $\M(a)$ in $\H(b)$ for the case when pairs $(b,a)$ are rational has been studied
for example   in\cite{CR}, \cite{FHR1}, \cite{Sarason5}, and
\cite{LN}. Also in \cite{LMT} this space has been described for
concrete nonextreme functions $b$ that are not rational. Analogous
result to that stated in Theorem 1 has been obtained in \cite{FHR1}
for rational pairs (or their positive powers) where the
corresponding point $z_0\in
\partial\D$ is a zero of order $m-1$ of the rational function $a$.
One can easily check   that in such a case there  exists a
$\lambda\in\partial\D$ for which $F_{\lambda}{(1-\bar{z}_0z)^{m}}\in
H^2$ and the hypotheses of Theorem 1 are satisfied.
\end{rems}

\end{document}